\documentclass[11pt]{amsart}

\pdfoutput=1
\usepackage{amssymb}
\usepackage{graphicx}
\usepackage{amsthm}
\usepackage{amscd}
\usepackage{float}
\usepackage{mathtools}
\usepackage{amssymb}
\usepackage{amsmath}
\usepackage{graphicx}
\usepackage{pict2e}
\usepackage[font=scriptsize]{caption}
\usepackage{graphicx}
\usepackage{tikz}
\usetikzlibrary{chains}

\tikzset{node distance=2em, ch/.style={circle,draw,on chain,inner sep=2pt},chj/.style={ch,join},every path/.style={shorten >=4pt,shorten <=4pt},line width=1pt,baseline=-1ex}

\usepackage{graphicx}

\usepackage{array}
\usepackage{verbatim}
\usepackage{color}
\usepackage[utf8]{inputenc}

\newtheorem{corollary}{Corollary}
\newtheorem*{corollary*}{Korollar}
\newtheorem*{notation*}{Notation}

\newtheorem*{dank*}{Danksagung}
\newtheorem{lemma}{Lemma}
\newtheorem*{example*}{Example}
\newtheorem{theorem}{Theorem}
\newtheorem{proposition}{Proposition}
\theoremstyle{remark}
\newtheorem{remark}{Remark}
\theoremstyle{theorema}

\newtheorem*{theoremb}{Conjecture}

\newtheorem*{theoremd}{Acknowledgement}

\theoremstyle{definition}
\newtheorem{definition}{Definition}

\DeclareMathOperator{\rank}{rank}

\title[Combinatorics of free and simplicial line arrangements]
{Combinatorics of free and simplicial line arrangements}

\author{David~Geis}
\email{davidgeis@web.de}

\begin{document}

\begin{abstract}
We study the combinatorics of pseudoline arrangements in the real projective plane. Our focus lies on two classes of arrangements: simplicial arrangements and arrangements whose characteristic polynomials have only real roots. We derive inequalities involving the $t$-vectors of the arrangements in consideration. As application, we obtain some finiteness and classification results. Moreover, we are able to prove the Dirac Motzkin Conjecture for real pseudoline arrangements whose charateristic polynomials split over $\mathbb{R}$.
\end{abstract}


\maketitle

\begin{section}{Introduction}
Arrangements of pseudolines in the real projective plane are classical objects of study in combinatorics and geometry. In this note, we are interested in the combinatorics of two classes of such arrangements: arrangements whose associated cell decompositions of $\mathbb{P}^2(\mathbb{R})$ are triangulations (so called \textit{simplicial arrangements}) and arrangements whose \textit{characteristic polynomials have only real roots}. By Terao's Factorization Theorem, the latter class includes all pseudoline arrangements which originate from a \textit{free} hyperplane arrangement.  

Despite some major progress (see for instance the papers \cite{Cuntz}, \cite{Cuntz3}, \cite{Cuntz4}, \cite{cuntz_muecksch}), a complete classification of simplicial arrangements still remains an open problem. However, there is a catalogue published by Gr\"unbaum (see \cite{Gruenbaum}), listing almost all currently known isomorphism classes of \textit{stretchable} arrangements. Since then, only four additional arrangements have been discovered (see the paper \cite{Cuntz}). 

The current belief is that -up to finitely many corrections- the given catalogue is complete. In this paper we collect some more evidence for this belief. In particular, we show that Gr\"unbaum's catalogue contains all arrangements which are free and whose vertices have weight bounded by four (see Corollary \ref{hoechstens 4er}). Similarly, we prove that a free simplicial pseudoline arrangement whose vertices have weight bounded by five consists of at most $40$ lines (see Corollary \ref{hoechstens fuenfer}). This implies that there is only a finite number of such arrangements possibly missing in Gr\"unbaum's catalogue.  

Our techniques also allow us to prove finiteness results for line arrangements which are not necessarily simplicial. For instance, we prove that there are only finitely many isomorphism classes of free line arrangements in $\mathbb{P}^2(\mathbb{R})$ whose vertices have weight bounded by five (see Theorem \ref{no simp theo}). 

Motivated by the paper \cite{greentao}, we also study the Dirac Motzkin Conjecture in the context of pseudoline arrangements whose characteristic polynomials split over $\mathbb{R}$. We are able to resolve the conjecture completely in this setup, including a classification of all extremal examples (see Theorem \ref{dirac theorem}).

Moreover, we prove a combinatorial analogue of the fact that hyperplane arrangements in $\mathbb{R}^3$ having isometric chambers are Coxeter arrangements (see Theorem \ref{cox theorem} and the paper \cite{Ehrenborg}).

\begin{theoremd}
This paper is a result of my stay at the Institut f\"ur Algebra, Zahlentheorie und Diskrete Mathematik at Leibniz Universit\"at Hannover.
I wish to thank Michael Cuntz for many helpful discussions. Moreover, I thank Piotr Pokora for some comments on an earlier version of this paper. Finally, I wish to thank the Deutsche Forschungsgemeinschaft (DFG) for their financial support.  
\end{theoremd}

\end{section}

\begin{section}{Definitions and well known results}
We start with the definition of pseudoline arrangements. For more background on the general concept of oriented matroids, we refer the reader to \cite{matroids}. For us, the following will suffice.
\begin{definition}
i) An arrangement of pseudolines is a finite set $\mathcal{A}$ of $n \geq 3$ smooth closed curves in $\mathbb{P}^2(\mathbb{R})$ such that the following conditions are satisfied: \begin{itemize}
 \item Curves in $\mathcal{A}$ do not intersect themselves.
 \item Different curves in $\mathcal{A}$ meet each other transversally in precisely one point.
 \item We have $\bigcap_{\ell \in \mathcal{A}} \ell = \lbrace \rbrace$, i.e. $\mathcal{A}$ is not a pencil arrangement.  
\end{itemize} 
Any arrangement $\mathcal{A}$ induces a cell decomposition of the real projective plane. We denote by $\mathcal{K}(\mathcal{A})$ the set of $2$-cells in said decomposition. These are usually called \textit{chambers}. Similarly, the $1$-cells are called \textit{edges} or \textit{segments} while the $0$-cells are called \textit{vertices}. We write $f^\mathcal{A}_0, f^\mathcal{A}_1, f^\mathcal{A}_2$ for the number of vertices, edges, chambers of $\mathcal{A}$ respectively. Two arrangements $\mathcal{A}, \mathcal{A}^\prime$ are called \textit{isomorphic}, if the corresponding cell decompositions are isomorphic; in this case, if $\mathcal{A}$ is an arrangement of straight lines, then $\mathcal{A}^\prime$ is called \textit{stretchable}. \\ 
ii) Let $\mathcal{A}$ be an arrangement of pseudolines in $\mathbb{P}^2(\mathbb{R})$. Then $\mathcal{A}$ is called \textit{simplicial} if every $C \in \mathcal{K}(\mathcal{A})$ is bounded by precisely three lines $\ell_1, \ell_2, \ell_3 \in \mathcal{A}$.  \\
iii) Let $n:=|\mathcal{A}|$ and choose a labelling $\mathcal{A}=\lbrace \ell_1, \ell_2, ..., \ell_n \rbrace$ for the lines of $\mathcal{A}$. We will associate with $\mathcal{A}$ a geometric lattice $L:=L_\mathcal{A}$. For this, define sets $L_0, L_1, L_2, L_3$  as follows: \begin{itemize}
 \item $L_0$ consists of the single element $\lbrace  \rbrace$.
 \item $L_1$ consists of the elements $\lbrace i \rbrace$ for $1 \leq i \leq n$.
 \item For each $v \in \mathbb{P}^2(\mathbb{R})$ we set $I_v:=\lbrace i \mid 1 \leq i \leq n,v \in \ell_i \rbrace$. Now we define $L_2:= \lbrace I_v \mid v \in \mathbb{P}^2(\mathbb{R}), |I_v| \geq 2 \rbrace$  
 \item $L_3$ consists of the single element $\lbrace 1,2, ..., n \rbrace$.
\end{itemize}
Now define $L:=L_0 \cup L_1 \cup L_2 \cup L_3$. It is immediate that $L$ is a poset via setwise inclusion. We have a natural rank function $r$ on $L$, which is described as follows: $r(X)=i$ if and only if $X \in L_i$. Together with this rank function, the poset $L$ becomes a geometric lattice. It is clear that different labellings for the lines of $\mathcal{A}$ will lead to isomorphic lattices. Therefore, we may define the \textit{characteristic polynomial of} $\mathcal{A}$ as the characteristic polynomial of $L$, i.e. we set $\chi(\mathcal{A},t):=\sum_{X \in L} \mu(\lbrace \rbrace, X) t^{3-\rank(X)}$, where $\mu$ denotes the M\"obius Function of $L$.
\end{definition}

We continue by introducing the main tool of this paper, the so called \textit{t-vector} associated to an arrangement of pseudolines.

\begin{definition}
Let $\mathcal{A}$ be an arrangement of pseudolines in $\mathbb{P}^2(\mathbb{R})$. \\
i) For each $v \in \mathbb{P}^2(\mathbb{R})$ we define $w_{\mathcal{A}}(v):=|\lbrace \ell \in \mathcal{A} \mid v \in \ell \rbrace| \in \mathbb{N}_{\geq 0}$ and call it the \textit{weight} of $v$ (with respect to $\mathcal{A}$). Points $v \in \mathbb{P}^2(\mathbb{R})$ such that $w_{\mathcal{A}}(v)=2$ are sometimes called \textit{double points (of} $\mathcal{A}${)}. Similarly, points $v$ with $w_{\mathcal{A}}(v)=3$ are called \textit{triple points (of} $\mathcal{A}${)}. Let $C \in \mathcal{K}(\mathcal{A})$ be a chamber. We associate with $C$ a graph $\Gamma^C$ defined as follows: the vertices of $\Gamma^C$ are given by the lines of $\mathcal{A}$ which bound $C$ and two vertices $\ell, \ell^\prime$ are connected by an edge with weight $w_{\mathcal{A}}(\ell \cap \ell^\prime)$ if and only if $w_{\mathcal{A}}(\ell \cap \ell^\prime) \geq 3$.  \\ 
ii)  For $2 \leq i \leq |\mathcal{A}|$ we define $t^\mathcal{A}_i :=| \lbrace v \in \mathbb{P}^2(\mathbb{R}) \mid w_{\mathcal{A}}(v)=i \rbrace|$. This means that $t^\mathcal{A}_i$ is the number of points in $\mathbb{P}^2(\mathbb{R})$ contained in precisely $i$ lines of $\mathcal{A}$. The vector $t^\mathcal{A} \in \mathbb{Z}^{|\mathcal{A}|-1}$ whose $i-1$-th component is given by $t^\mathcal{A}_{i}$ is called the \textit{t-vector} of $\mathcal{A}$. Moreover, if $j$ is maximal with the property $t^\mathcal{A}_j>0$, then we say that the \textit{multiplicity} of $\mathcal{A}$ is $j$ and we write $m(\mathcal{A}):=j$.    \\
iii) Assume that $n:=|\mathcal{A}| \geq 3$. If $\mathcal{A}$ is an arrangement such that $t^\mathcal{A}_2=n-1$ and $t^\mathcal{A}_{n-1}=1$ then $\mathcal{A}$ is called a \textit{near pencil (arrangement)}. Near pencil arrangements are usually considered trivial.
\end{definition}

The following lemma provides a collection of well known results on the $t$-vector of a pseudoline arrangement.

\begin{lemma} \label{basic comb t vector}
Let $\mathcal{A}$ be an arrangement of $n$ pseudolines in $\mathbb{P}^2(\mathbb{R})$. Then the following statements hold: \begin{align}
\sum_{i \geq 2} \binom{i}{2} t^\mathcal{A}_i&= \binom{n}{2}, \\
1 + \sum_{i \geq 2} (i-1)t^\mathcal{A}_i&=f^\mathcal{A}_2, \\
\sum_{i \geq 2} t^\mathcal{A}_i&=f^\mathcal{A}_0, \\
\sum_{i \geq 2} i t^\mathcal{A}_i&=f^\mathcal{A}_1, \\
3 + \sum_{i \geq 4} (i-3)t^\mathcal{A}_i  &\leq t^\mathcal{A}_2.
\end{align}
\end{lemma}

\begin{proof}
Equation (1) follows from counting pairs of lines in two different ways. Equality (3) holds by definition. In order to obtain (4), observe that every edge is contained in precisely one line. Moreover, the number of vertices on a given line coincides with the number of edges contained in said line. We obtain (4) by another double counting argument. Using the formula for the Euler characteristic, we see that (2) is a consequence of (3) and (4). Finally, inequality (5) follows from the observation that every chamber of $\mathcal{A}$ is bounded by at least three lines $\ell_1, \ell_2, \ell_3 \in \mathcal{A}$.
\end{proof}

\begin{remark} \label{near pencil remark}
Inequality (5) is also known as ``Melchior's inequality'' (see \cite{melch}).
An arrangement $\mathcal{A}$ is simplicial if and only if we have equality in (5). Indeed, $\mathcal{A}$ is simplicial if and only if $3 f^\mathcal{A}_2 = 2 f^\mathcal{A}_1$. The claim now follows using relations (2) and (4).
\end{remark}

\end{section}

\begin{section}{Simplicial arrangements}
The main goal of this section is to provide inequalities involving the $t$-vectors of simplicial pseudoline arrangements in $\mathbb{P}^2(\mathbb{R})$. If not explicitly stated otherwise, then throughout the entire section $\mathcal{A}$ will always denote a simplicial pseudoline arrangement consisting of $n$ lines in $\mathbb{P}^2(\mathbb{R})$. If not stated otherwise, it is assumed that $t^\mathcal{A}_i=0$ for $i \in \lbrace n-1,n \rbrace$, i.e. $\mathcal{A}$ is not a near pencil arrangement. Moreover, isomorphism classes of simplicial arrangements are denoted in the same way as in \cite{Gruenbaum}.

\begin{subsection}{Upper and lower bounds for $t^\mathcal{A}_2$ and $t^\mathcal{A}_3$}

We will be mainly interested in (lower and upper) bounds for the numbers $t^\mathcal{A}_2, t^\mathcal{A}_3$. To get started, we give the following lemma which gives quite strong restrictions on the distribution of double points produced by a simplicial pseudoline arrangement.

\begin{lemma} \label{near pencil lemma}
a) The closure of any given chamber $C \in \mathcal{K}(\mathcal{A})$ contains at most one double point of $\mathcal{A}$. \\
b) We have the tight estimate $4t^\mathcal{A}_2 \leq  f^\mathcal{A}_2$. Equality holds if and only if $\overline{C}$ contains a double point for every $C \in \mathcal{K}(\mathcal{A})$.
\end{lemma}

\begin{proof}
a) Assume that there exists $C \in \mathcal{K}(\mathcal{A})$ such that $\overline{C}$ contains two double points $v_1, v_2$. Let $\ell \in \mathcal{A}$ be the line containing both $v_1, v_2$ and let $\ell_1, \ell_2 \in \mathcal{A}$ be the lines meeting $\ell$ in $v_1, v_2$ respectively. 
Let $C^\prime \in \mathcal{K}(\mathcal{A})$ be the chamber adjacent to $C$ via $\ell_2$. Observe that both $\ell$ and $\ell_2$ are walls of $C^\prime$ because $w_{\mathcal{A}}(v_2)=2$. Denote by $\ell_3$ the line supporting the third wall of $C^\prime$ and set $v_3:=\ell_3 \cap \ell$. In particular, the points $v_2,v_3,v:=\ell_1 \cap \ell_2$ are vertices of $\mathcal{A}$ contained in $\overline{C^\prime}$. This implies that $\ell_3$ passes through $v$. Iterating this argument $n-3$ times in total shows that $\mathcal{A}$ is a near pencil arrangement, contradicting our initial assumptions about $\mathcal{A}$.   \\
b) By part a) and double counting, we have $4 t^\mathcal{A}_2 \leq f^\mathcal{A}_2$. This inequality is tight for the arrangement $A(13,2)$ shown in Figure \ref{a_13_2 bild}. The claim about equality also follows from part a).  
\end{proof}

\begin{remark} 
a) If $\mathcal{A}$ is a near pencil arrangement, then $t^\mathcal{A}_2=n-1$ and $f^\mathcal{A}_2=2n-2$. So in this case one has $f^\mathcal{A}_2<4 t^\mathcal{A}_2$. \\
b) Further examples of arrangements with $4 t^\mathcal{A}_2 =f^\mathcal{A}_2$ are for instance given by the (straight) line arrangements corresponding to finite reflection groups in $\mathbb{R}^3$. 
\end{remark}

From the last lemma we may deduce the following interesting analogue of Melchior's inequality which is valid only for non-trivial simplicial arrangements:

\begin{corollary} \label{simplicial melchior}
We have the following analogue of Melchior's inequality for the number $t^\mathcal{A}_3$: $$t^\mathcal{A}_3 \geq 4 + \sum_{i \geq 5} (i-4) t^\mathcal{A}_i.$$
 Equality holds if and only if $\overline{C}$ contains a double point for every $C \in \mathcal{K}(\mathcal{A})$.
\end{corollary}

\begin{proof}
By Lemma \ref{near pencil lemma}, part b) we have $f^\mathcal{A}_2 \geq 4 t^\mathcal{A}_2$. By Remark \ref{near pencil remark} we have equality in (5), therefore $$ -2 + 2 \sum_{i \geq 3} t^\mathcal{A}_i=f^\mathcal{A}_2 - 2 t^\mathcal{A}_2 \geq 2 t^\mathcal{A}_2=2\left(3 + \sum_{i \geq 4} (i-3)t^\mathcal{A}_i \right).$$ After rearranging terms we obtain the desired inequality. 

Now assume that $t^\mathcal{A}_3=4 + \sum_{i \geq 5} (i-4) t^\mathcal{A}_i$. This implies $t^\mathcal{A}_2=-1 + \sum_{i \geq 3} t^\mathcal{A}_i$, which yields $4 t^\mathcal{A}_2=f^\mathcal{A}_2$.  
\end{proof}

\begin{figure}[H]
\begin{center}
\setlength{\unitlength}{0.7pt}
\begin{picture}(400,300)(90,0) 
\Line(220,0)(220,300)
\Line(340,0)(340,300)

\Line(120,0)(420,300)
\put(270,115){
$C_1$}
\put(310,155){
$C_2$}
\put(270,190){
$C_3$}

\put(225,155){
$C_4$}

\put(270,75){
$C_5$}

\put(350,155){
$C_6$}

\put(270,235){
$C_7$}

\put(185,155){
$C_8$}

\Line(140,300)(440,0)

\Line(100,100)(460,100)
\Line(100,220)(460,220)

\thicklines
\strokepath
\end{picture}\\
The situation of Lemma \ref{jeder zweier hoechstens drei dreier}.   \caption{\label{Lemma bild}} 
\end{center}
\end{figure}

\begin{lemma} \label{jeder zweier hoechstens drei dreier}
 Assume that $\mathcal{A}$ has some vertex $v$ of weight two such that every neighbour of $v$ has weight three. Then $n\leq7$, i.e. $\mathcal{A}$ is one of the arrangements $A(6,1), A(7,1)$.
\end{lemma}

\begin{proof}
The situation of the lemma is shown in Figure \ref{Lemma bild}. The assumptions imply that any further line of $\mathcal{A}$ must not intersect the chambers $C_1, C_2, ..., C_8$. Thus, only one more line may be added to $\mathcal{A}$. This proves the claim.  
\end{proof}

\begin{remark}
If $n \geq 8$ then there exists a chamber containing at most one vertex of weight three. This follows immediately from Lemma \ref{jeder zweier hoechstens drei dreier}.
\end{remark}

In order to obtain an upper bound for the number of double points in terms of $n$, it suffices by part b) of Lemma \ref{near pencil lemma} to give an upper bound for the number of chambers. This is done in the following Proposition.

\begin{proposition} \label{2er bound}
We have the tight estimate $t^\mathcal{A}_2 \leq \frac{\binom{n}{2}+6}{7}$.  
\end{proposition}

\begin{proof}
By part b) of Lemma \ref{near pencil lemma} we have $4 t^\mathcal{A}_2 \leq f^\mathcal{A}_2$. Further, using the relations given in Lemma \ref{basic comb t vector} together with Remark \ref{near pencil remark}, one can check that \begin{align*}
f^\mathcal{A}_2=\binom{n}{2} +1 - \sum_{i \geq 3} \binom{i-1}{2}t^\mathcal{A}_i \leq \binom{n}{2} +1 -t^\mathcal{A}_3 - 3 t^\mathcal{A}_2+9.
\end{align*}
Using $t^\mathcal{A}_3 \geq 4$, we may conclude that $4 t^\mathcal{A}_2 \leq f^\mathcal{A}_2 \leq \frac{n^2-n+12}{2} - 3 t^\mathcal{A}_2$, proving the first claim. The arrangement $A(13,2)$ depicted in Figure \ref{a_13_2 bild} is an example for which the given bound is tight.
\end{proof}

\begin{remark} \label{sharpness bei 4} 
If $t^\mathcal{A}_2=\frac{\binom{n}{2}+6}{7}$ then we have equality in the chain of inequalities $$4 t^\mathcal{A}_2 \leq f^\mathcal{A}_2 \leq \frac{n^2-n}{3} - \frac{2}{3} t^\mathcal{A}_2 +4.$$ This in turn implies that $\mathcal{A}$ has multiplicity at most four. Moreover, we observe that one has $$t^\mathcal{A}_3=4 \Leftrightarrow t^\mathcal{A}_4=\frac{\binom{n}{2}-15}{7} \Leftrightarrow t^\mathcal{A}_2=\frac{\binom{n}{2}+6}{7}.$$ 
\end{remark}

\begin{figure} 
\begin{center}
\setlength{\unitlength}{0.7pt}
\begin{picture}(400,300)(90,0) 
\Line(220,0)(220,300)
\Line(280,0)(280,300)
\Line(340,0)(340,300)

\Line(120,0)(420,300)
\Line(180,0)(460,280)
\Line(100,40)(360,300)

\Line(200,300)(460,40)
\Line(140,300)(440,0)
\Line(100,280)(380,0)

\Line(100,100)(460,100)
\Line(100,160)(460,160)
\Line(100,220)(460,220)

\put(480,280){\begin{Large}
$\infty$\end{Large}}

\thicklines
\strokepath
\end{picture}\\
The arrangement of type $A(13,2)$. The line at infinity is contained in the arrangement, indicated by the symbol ``$\infty$''.   \caption{\label{a_13_2 bild}} 
\end{center}
\end{figure}

In the following, we want to establish an upper bound for $\min(t^\mathcal{A}_2, t^\mathcal{A}_3)$ and a lower bound for $\max(t^\mathcal{A}_2, t^\mathcal{A}_3)$. In order to do this we use a little lemma which may be interesting in its own right. 

\begin{lemma} \label{t2+t3} The following statements are true: \\
a) $2 (t^\mathcal{A}_2 + 2) \leq 4 + \frac{f^\mathcal{A}_2}{2}  + \sum_{i\geq 5} (i-4) t^\mathcal{A}_i  = 2 t^\mathcal{A}_2 + t^\mathcal{A}_3 \leq 2 t^\mathcal{A}_2 + t^\mathcal{A}_3 + \frac{t^\mathcal{A}_4}{3}\leq \frac{\binom{n}{2}}{3} + 5$.  \\
b) If $t^\mathcal{A}_i=0$ for $i>6$ then we have $2 t^\mathcal{A}_2 + t^\mathcal{A}_3 + \frac{t^\mathcal{A}_4}{3} = \frac{\binom{n}{2}}{3} + 5$. In particular, we cannot have $t^\mathcal{A}_4 \equiv 2 \ (\text{mod } 3 )$ for such an arrangement.
\end{lemma}

\begin{proof}
a) Equation (1) from Lemma \ref{basic comb t vector} gives $3 t^\mathcal{A}_3= \binom{n}{2} -t^\mathcal{A}_2 - \sum_{i \geq 4} \binom{i}{2} t^\mathcal{A}_i$. Moreover, for $i \geq 5$ we always have $\binom{i}{2} \geq 5 (i-3)$ and so we conclude that $
6 t^\mathcal{A}_4 + \sum_{i \geq 5} \binom{i}{2} t^\mathcal{A}_i \geq  t^\mathcal{A}_4 + 5 \sum_{i \geq 4} (i-3) t^\mathcal{A}_i=5 t^\mathcal{A}_2 +t^\mathcal{A}_4 -15.
$ From this, it follows that $3 t^\mathcal{A}_3 \leq \binom{n}{2} -6 t^\mathcal{A}_2 - t^\mathcal{A}_4+15$, 
proving the upper bound for $2 t^\mathcal{A}_2 + t^\mathcal{A}_3+ \frac{t^\mathcal{A}_4}{3}$.

 Next we show that $2 t^\mathcal{A}_2 + t^\mathcal{A}_3 = 4 + \frac{f^\mathcal{A}_2}{2} + \sum_{i\geq 5} (i-4) t^\mathcal{A}_i$: observe that
$t^\mathcal{A}_2 =3 + \sum_{i \geq 4} (i-3) t^\mathcal{A}_i =  4 + \sum_{i \geq 4} t^\mathcal{A}_i - 1 + \sum_{i\geq 5} (i-4) t^\mathcal{A}_i$.
 Adding $t^\mathcal{A}_2 + t^\mathcal{A}_3$ on both sides of the last equation gives the desired equality.

Finally, the inequality $2 (t^\mathcal{A}_2 + 2) \leq 4 + \frac{f^\mathcal{A}_2}{2} + \sum_{i\geq 5} (i-4) t^\mathcal{A}_i$ follows from part b) of Lemma \ref{near pencil lemma}. \\
b) If $t^\mathcal{A}_i=0$ for $i>6$ then $ t^\mathcal{A}_2+3 t^\mathcal{A}_3 + 6 t^\mathcal{A}_4+ 10 t^\mathcal{A}_5 + 15 t^\mathcal{A}_6=\binom{n}{2}$, using equation (1) from Lemma \ref{basic comb t vector}. By simpliciality of $\mathcal{A}$ we have $t^\mathcal{A}_2=3+t^\mathcal{A}_4 + 2 t^\mathcal{A}_5 + 3 t^\mathcal{A}_6$ (see Remark \ref{near pencil remark}). We conclude that $6 t^\mathcal{A}_2+ 3 t^\mathcal{A}_3 + t^\mathcal{A}_4 - 15= \binom{n}{2}$. It follows $n^2-n-2t^\mathcal{A}_4 \equiv 0 \ (\text{mod } 3 )$. As the polynomial $X^2-X+2$ is irreducible over the finite field $\mathbb{F}_3$, it follows that $t^\mathcal{A}_4 \not \equiv 2 \ ( \text{mod } 3 )$. This completes the proof.
\end{proof}

\begin{corollary} \label{min max bound}
We have $\min(t^\mathcal{A}_2,t^\mathcal{A}_3) \leq \frac{n^2-n+30}{18}$ and $\max(t^\mathcal{A}_2,t^\mathcal{A}_3) > \frac{f^\mathcal{A}_2}{6}$.
\end{corollary}

\begin{proof}
Assume that $\max_{i \geq 2} t^\mathcal{A}_i= t^\mathcal{A}_3$. Then by Lemma \ref{t2+t3} we have \begin{align*}
3 t^\mathcal{A}_2 &\leq 2 t^\mathcal{A}_2 + t^\mathcal{A}_3 \leq \frac{\binom{n}{2}}{3} + 5, \\
\frac{f^\mathcal{A}_2}{2} &<  2 t^\mathcal{A}_2 + t^\mathcal{A}_3 \leq 3 t^\mathcal{A}_3.
\end{align*} This proves the claim in case $\max_{i \geq 2} t^\mathcal{A}_i= t^\mathcal{A}_3$. The case $\max_{i \geq 2} t^\mathcal{A}_i= t^\mathcal{A}_2$ is dealt with similarly.
\end{proof}

We close this subsection with a theorem which asserts that for stretchable(!) arrangements, at least one of $t^\mathcal{A}_2, t^\mathcal{A}_3$ is quadratic in $|\mathcal{A}|$. For this we need one more result:

\begin{proposition} \label{inf series prop}
Let $m:=\left \lfloor{\frac{|\mathcal{A}|}{2}}\right \rfloor$. Then we have $t^\mathcal{A}_i=0$ for all $i> m$. Moreover, we always have $t^\mathcal{A}_m \leq 1$.
\end{proposition}

\begin{proof}
Suppose that there was some $i>m$ such that $t^\mathcal{A}_i>0$. Pick a vertex $v$ of weight $i$ and denote the set of lines passing through $v$ by $L_v$. Then there are $2i$ chambers $K_1, ..., K_{2i}$ having $v$ as a vertex. Each of these chambers has precisely one wall supported by a line not contained in $L_v$. As $|\mathcal{A} \setminus L_v|<m$, we conclude that there must be some $\ell \in \mathcal{A} \setminus L_v$ such that $\ell$ is a wall in three neighbouring chambers $K_{j_1}, K_{j_2}, K_{j_3}$. But then $\ell$ contains a segment bounded by two vertices of weight two, contradicting our initial assumption that $\mathcal{A}$ is not a near pencil arrangement. This proves the first claim. Next we show that $t_m \leq 1$. Suppose that $t_m>1$. Then clearly $t_m=2$ and we denote the two vertices of weight $m$ by $v_1$ and $v_2$. Then any line of $\mathcal{A}$ not passing through both $v_1$ and $v_2$ contains a segment bounded by two vertices of weight two, another contradiction. This completes the proof.  
\end{proof}


Using what we have established so far together with results from \cite{Langer} and \cite{shnu}, we are now in a position to prove the following theorem:

\begin{theorem} \label{t2 t3 quadratisch thm}
The following statements hold: \\
a) We have $\max(t^\mathcal{A}_2, t^\mathcal{A}_3)> \frac{ n^2-n+2m(\mathcal{A})}{3 (m(\mathcal{A})+3)}$. \\
b) If $\mathcal{A}$ is stretchable, then $\max(t^\mathcal{A}_2, t^\mathcal{A}_3) > \frac{n^2+3n}{27}$.
\end{theorem}

\begin{proof}
a) By Corollary \ref{min max bound} and \cite[Theorem 1]{shnu} we have $$6 \max(t^\mathcal{A}_2, t^\mathcal{A}_3) >f^\mathcal{A}_2 \geq \frac{2 n^2-2n+4m(\mathcal{A})}{m(\mathcal{A})+3}.$$ This proves the claim. \\
b) Observe that by Proposition \ref{inf series prop} we have $t^\mathcal{A}_i=0$ for $i> \frac{n}{2}$. Thus, we may apply \cite[Proposition 11.3.1]{Langer} to obtain the following estimate: \begin{align*}
f^\mathcal{A}_1=\sum_{i \geq 2} i t^\mathcal{A}_i \geq  \frac{n^2+3n}{3}. 
\end{align*}
As $\mathcal{A}$ is simplicial we have $3 f^\mathcal{A}_2=2 f^\mathcal{A}_1$ (see Remark \ref{near pencil remark}), hence using Corollary \ref{min max bound} we obtain the inequality \begin{align*}
\max(t^\mathcal{A}_2, t^\mathcal{A}_3) > \frac{f^\mathcal{A}_2}{6}  \geq \frac{n^2+3n}{27},
\end{align*} finishing the proof of part b).    
\end{proof}
\begin{remark} \label{regionale schranke}
Note that \cite[Proposition 11.3.1]{Langer} is a result on \textit{linear} line arrangements in $\mathbb{P}^2(\mathbb{C})$. However, by complexification every arrangement $\mathcal{A}$ in $\mathbb{P}^2(\mathbb{R})$ yields an arrangement $\mathcal{A}_\mathbb{C}$ in $\mathbb{P}^2(\mathbb{C})$ with the same $t$-vector.
\end{remark}

\end{subsection}

\begin{subsection}{An application: a combinatorial characterization of spherical Coxeter arrangements in $\mathbb{R}^3$}
In this subsection we prove that a pseudoline arrangement $\mathcal{A}$ is combinatorially isomorphic to a spherical Coxeter arrangement if and only if there exists a suitable connected graph $\Gamma$ such that $\Gamma^C \cong \Gamma$ for any chamber $C \in \mathcal{K}(\mathcal{A})$. This can be regarded as a combinatorial analogue of the theorem which asserts that spherical rank three Coxeter arrangements are characterized as those arrangements in $\mathbb{R}^3$ which have isometric chambers (see \cite{Ehrenborg}).

\begin{lemma} \label{cox_lemma}
Let $\mathcal{A}$ be an arbitrary (i.e. not necessarily simplicial) arrangement.
Assume that there is a connected graph $\Gamma$ such that $\Gamma^C \cong \Gamma$ for every $C \in \mathcal{K}(\mathcal{A})$. Then $\mathcal{A}$ is simplicial and there exists $x \in \mathbb{N}$ such that \begin{align*}
\Gamma=\begin{tikzpicture}
\draw[fill=black]
(0,0)
circle [radius=.1] node [above] {} --
(1,0)
circle [radius=.1] node [above] {}
(1,0) --++ (0:1)
circle [radius=.1] node [above] {}
node [midway,above] {$x$}
;
\end{tikzpicture}.
\end{align*}
\end{lemma}

\begin{proof}
As $\Gamma$ is assumed to be connected, we see that $\mathcal{A}$ is not a near pencil arrangement. By a result of Levi (see for instance Theorem 6.5.2 in \cite{matroids}), we know that $\mathcal{A}$ contains at least $n$ chambers which are triangles. By assumption, this implies that every chamber of $\mathcal{A}$ must be a triangle and hence the arrangement $\mathcal{A}$ is necessarily simplicial.
We have $t^\mathcal{A}_2\geq 3$ by inequality (5) from Lemma \ref{basic comb t vector}. Moreover, we have $t^\mathcal{A}_3 \geq 4$ by Lemma 2. This proves the claim. 
\end{proof}

\begin{proposition} \label{cox_prop}
Fix $x \in \mathbb{N}$ and let $\mathcal{A}$ be a simplicial arrangement such that $$\Gamma^C \cong \begin{tikzpicture}
\draw[fill=black]
(0,0)
circle [radius=.1] node [above] {} --
(1,0)
circle [radius=.1] node [above] {}
(1,0) --++ (0:1)
circle [radius=.1] node [above] {}
node [midway,above] {$x$}
;
\end{tikzpicture}$$ for every chamber $C \in \mathcal{K}( \mathcal{A}) $. 
Then $x \in \lbrace 3, 4, 5 \rbrace$. If $x=3$ then $\mathcal{A}$ is of type $A(6,1)$, if $x=4$ then $\mathcal{A}$ is of type $A(9,1)$ and if $x=5$ then $\mathcal{A}$ is of type $A(15,1)$. In particular, $\mathcal{A}$ is isomorphic to a spherical Coxeter arrangement. 
\end{proposition}

\begin{proof}
Suppose that $x=3$, so $\mathcal{A}$ has only vertices of weight two or three. Then by Lemma \ref{jeder zweier hoechstens drei dreier} the arrangement $\mathcal{A}$ is of type $A(6,1)$ or $A(7,1)$. Since the arrangement $A(7,1)$ contains a chamber having only vertices of weight three, it follows that $\mathcal{A}$ is of type $A(6,1)$. 

Now assume that $x>3$. As $\mathcal{A}$ is simplicial, we have  \begin{align}
t^\mathcal{A}_2 - (x-3)t^\mathcal{A}_x -3=0.
\end{align} 
On the other hand, we have the following identities  \begin{align}
t^\mathcal{A}_2 + \binom{3}{2} t^\mathcal{A}_3 + \binom{x}{2} t^\mathcal{A}_x - \binom{n}{2}&=0, \\
2 t^\mathcal{A}_2 - 3 t^\mathcal{A}_3&=0, \\
3 t^\mathcal{A}_3-x t^\mathcal{A}_x&=0. 
\end{align} 
Regarding $x$ as a variable, we consider the function field $\mathbb{F}:=\mathbb{Q}(x)$ and think of $n, t^\mathcal{A}_2, t^\mathcal{A}_3, t^\mathcal{A}_x$ as variables in a polynomial ring $R:=\mathbb{F}[n, t^\mathcal{A}_2, t^\mathcal{A}_3, t^\mathcal{A}_x]$. 

In $R$ we consider the ideal $I$ generated by the relations (6), (7), (8), (9) and we compute the following Gr\"obner basis for $I$: \begin{align*}
I=\left( 2 \binom{n}{2} + \frac{12x+6x^2}{x-6}, t^\mathcal{A}_2  + \frac{3x}{x-6}, t^\mathcal{A}_3 + \frac{2x}{x-6}, t^\mathcal{A}_x + \frac{6}{x-6}\right).
\end{align*}  
As $t^\mathcal{A}_2 > 0$ and $x \geq 4$ we infer that $x-6<0$, hence $4 \leq x \leq 5$. For $x=4$ we obtain $n=9$ and $t^\mathcal{A}=(6,4,3)$; if $x=5$ then $n=15$ and $t^\mathcal{A}=(15,10,0,6)$ (where trailing zeroes are omitted). Now we may use the results in \cite{Cuntz} to obtain the full statement.  
\end{proof}

Lemma \ref{cox_lemma} and Proposition \ref{cox_prop} now immediately give us the announced theorem:

\begin{theorem} \label{cox theorem}
Let $\mathcal{A}$ be an arbitrary (i.e. not necessarily simplicial) arrangement. Then the following statements are equivalent: \\
a)  There exists a connected graph $\Gamma$ such that $\Gamma^C \cong \Gamma$ for every $C \in \mathcal{K}(\mathcal{A})$.  \\
b) $\mathcal{A}$ is isomorphic to a spherical Coxeter arrangement.
\end{theorem}
\end{subsection}

\end{section}

\begin{section}{Arrangements whose characteristic polynomials have only real roots}
In this section, we study the combinatorics of pseudoline arrangements whose characteristic polynomials have only real roots. If not stated otherwise, then throughout the entire section $\mathcal{A}$ denotes a pseudoline arrangement consisting of $n$ lines such that $\chi(\mathcal{A},t)$ splits over $\mathbb{R}$. As in the last section, isomorphism classes of simplicial arrangements are denoted in the same way as in \cite{Gruenbaum} (this is relevant only for Subsections 4.2, 4.3). 

We begin with the following key lemma which allows us to give a nontrivial bound on $f^\mathcal{A}_2=|\mathcal{K}(\mathcal{A})|$ in terms of $n$. 

\begin{lemma} \label{poly lemma}
Set $m:=(n+1)^2-4 f^\mathcal{A}_2$. Then we have the following formula for the characteristic polynomial of $\mathcal{A}$: \begin{align*}
\chi(\mathcal{A},t)=t^3 - n t^2 + (f^\mathcal{A}_2-1)t + n - f^\mathcal{A}_2.
\end{align*} 
In particular, the roots of $\chi(\mathcal{A},t)$ are given by $1, \frac{n-1 + \sqrt{m} }{2}, \frac{n-1 - \sqrt{m} }{2}$ and we have the upper bound $f^\mathcal{A}_2 \leq \frac{(n+1)^2}{4}$.
\end{lemma}

\begin{proof}
Let $L$ denote the geometric lattice associated to $\mathcal{A}$ and denote its M\"obius Function by $\mu$. By abuse of notation we write $\mu(X):=\mu(\lbrace \rbrace,X)$ for $X \in L$. Then by definition we have 
 $\chi:=\chi(\mathcal{A},t)=\sum_{X \in L} \mu(X) t^{3-\rank(X)}$. If $X \in L$ has rank $2$ then $\mu(X)=|X|-1$ and if $X$ has rank $1$ then $\mu(X)=-1$. Further, $\mu(\lbrace \rbrace)=1$ and $\mu(\lbrace 1,2, ...,n \rbrace)=- \sum_{\lbrace 1,2, ..., n \rbrace \neq Y \in L} \mu(Y)$. Write $L_2$ for the subset of $L$ consisting of all elements of rank $2$. The claimed formula for $\chi$ then follows from the identity $\sum_{X \in L_2} |X|-1 = \sum_{i \geq 2} (i-1)t^\mathcal{A}_i$ by straightforward calculation. For this note that $\sum_{i \geq 2} (i-1)t^\mathcal{A}_i =f^\mathcal{A}_2-1$, as the Euler characteristic of $\mathbb{P}^2(\mathbb{R})$ is equal to one. In order to obtain the upper bound for $f^\mathcal{A}_2$, observe that by assumption all the roots of $\chi$ are real. Therefore $m \geq 0$ and the bound follows.
\end{proof}

\begin{remark}
Let $\mathcal{A}$ be a \textit{linear} arrangement. If $\mathcal{A}$ is free, then by Tearo's Factorization Theorem, all roots of $\chi(\mathcal{A},t)$ are integral and therefore real (see chapter 4 in \cite{hyperhyper}).
\end{remark}

\begin{subsection}{Arrangements having multiplicity at most five}
The result of this subsection is Theorem \ref{no simp theo}. It implies that a free linear arrangement whose multiplicity is bounded by five consists of at most $185$ lines. In particular, there are only finitely many isomorphism classes of such arrangements.  \\

We start with the following lemma which is similar to Theorem \ref{sechser struc} of the last subsection. In the proof we use a recent result from the paper \cite{shnu2}. 

\begin{lemma} \label{not simp}
Assume $n \geq 8$ and $m(\mathcal{A}) \leq 5$. Then the following is true: \begin{align}
\frac{t^\mathcal{A}_4}{3} +  t^\mathcal{A}_5 &\geq \frac{(n-5)^2-4}{24}, \\ 
t^\mathcal{A}_2 &\geq \frac{n^2-46n+233}{8} + 2 t^\mathcal{A}_4,  \\ 
\max(t^\mathcal{A}_4, t^\mathcal{A}_5) &\geq \frac{n^2 - 10 n + 21}{32}.
\end{align}
\end{lemma}

\begin{proof}
By part b) of \cite[Theorem 1]{shnu2} we have $t^\mathcal{A}_2 + \frac{3 t^\mathcal{A}_3}{2} \geq 8 + \frac{t^\mathcal{A}_4}{2} + \frac{5 t^\mathcal{A}_5}{2}$. As $3 t^\mathcal{A}_3=\binom{n}{2} - t^\mathcal{A}_2 - 6 t^\mathcal{A}_4 - 10 t^\mathcal{A}_5$ we may rewrite this as $\frac{t^\mathcal{A}_2}{2} + \frac{n^2-n}{4} - 3 t^\mathcal{A}_4 - 5 t^\mathcal{A}_5 \geq 8 + \frac{t^\mathcal{A}_4}{2} + \frac{5 t^\mathcal{A}_5}{2}$. It follows $t^\mathcal{A}_5 \leq \frac{n^2-n}{30} + \frac{t^\mathcal{A}_2}{15} - \frac{7 t^\mathcal{A}_4}{15} - \frac{16}{15}$. By Lemma \ref{poly lemma} we have $(n+1)^2 \geq 4 f^\mathcal{A}_2$. Equation (2) in Lemma \ref{basic comb t vector} yields $f^\mathcal{A}_2=1 + f^\mathcal{A}_1- f^\mathcal{A}_0=1 + t^\mathcal{A}_2 + 2 t^\mathcal{A}_3 + 3 t^\mathcal{A}_4 + 4 t^\mathcal{A}_5=\frac{n^2 - n + 3}{3} + \frac{t^\mathcal{A}_2}{3} - t^\mathcal{A}_4 - \frac{8 t^\mathcal{A}_5}{3} $. Now we use inequality (5) from Lemma \ref{basic comb t vector} to conclude that the estimate \begin{align*}
\frac{n^2}{4} + \frac{n}{2} + \frac{1}{4} \geq f^\mathcal{A}_2 \geq \frac{n^2-n+6}{3} - \frac{2 t^\mathcal{A}_4}{3} - 2 t^\mathcal{A}_5
\end{align*} holds. From this we deduce that $t^\mathcal{A}_5 \geq \frac{n^2}{24} - \frac{5 n}{12} + \frac{7}{8} - \frac{t^\mathcal{A}_4}{3}$, proving (10).

 We have thus established the following chain of inequalities: \begin{align*}
\frac{n^2 - 10n}{24} + \frac{7}{8} - \frac{t^\mathcal{A}_4}{3} \leq t^\mathcal{A}_5   \leq \frac{n^2-n}{30} + \frac{t^\mathcal{A}_2 - 16}{15} - \frac{7 t^\mathcal{A}_4}{15}.
\end{align*} This implies (11). In order to prove (12) we consider two cases. First assume that $t^\mathcal{A}_4 \leq t^\mathcal{A}_5$. Then by the above we know that $t^\mathcal{A}_5 \geq \frac{n^2 - 10n}{24} + \frac{7}{8} - \frac{t^\mathcal{A}_4}{3} \geq \frac{n^2 - 10n}{24} + \frac{7}{8} - \frac{t^\mathcal{A}_5}{3}$. From this we conclude that $t^\mathcal{A}_5 \geq \frac{n^2-10n+21}{32}$. The case $t^\mathcal{A}_4 \geq t^\mathcal{A}_5$ is dealt with similarly. This finishes the proof.
\end{proof}

With the last lemma we are ready to prove the main result of this subsection.

\begin{theorem} \label{no simp theo}
a) If $m(\mathcal{A}) \leq 4$, then $n \leq 19$. \\
b) If $m(\mathcal{A}) \leq 5$, then $n \leq 185$. 
\end{theorem}

\begin{proof}
a) Using \cite[Theorem 1]{shnu}, we obtain $$\frac{(n+1)^2}{4} \geq f^\mathcal{A}_2 \geq \frac{2 n^2-2n+4m(\mathcal{A})}{m(\mathcal{A})+3}.$$ For $m(\mathcal{A})=3$, this gives $\frac{(n+1)^2}{4} \geq \frac{2 n^2-2n+4 \cdot 3}{3+3}$. We conclude $3 \leq n \leq 7$. On the other hand, for $m(\mathcal{A})=4$ we obtain $\frac{(n+1)^2}{4} \geq \frac{2 n^2-2n+4 \cdot 4}{4+3}$. This yields $3 \leq n \leq 19$, completing the proof.
\\
b) We may assume that $n>7$. We have $\frac{(n+1)^2}{4} \geq f^\mathcal{A}_2$. Together with the first estimate in Lemma \ref{not simp} this yields \begin{align*}
\frac{(n+1)^2}{4} \geq f^\mathcal{A}_2 \geq 1 + t^\mathcal{A}_2 + 2 t^\mathcal{A}_3 + 4 (\frac{t^\mathcal{A}_4}{3}  +  t^\mathcal{A}_5) \geq 1 + t^\mathcal{A}_2 +  \frac{n^2-10n+21}{6}.
\end{align*} We conclude $\frac{n^2 + 26n - 51}{12} \geq t^\mathcal{A}_2$. But now the second estimate in Lemma \ref{not simp} gives us $\frac{n^2 + 26n - 51}{12}  \geq t^\mathcal{A}_2 \geq \frac{n^2-46n+233}{8}$. This implies $n \leq 185$, finishing the proof.
\end{proof}

\begin{remark}
If for $i \geq 6$ the numbers $t^\mathcal{A}_i$ do not grow too fast, then we can also give an upper bound for $n$. More precisely, for each $i \geq 6$ define $\Delta_i:=\frac{i^2-3i-10}{2}$ and assume that $t^\mathcal{A}_i \leq \alpha_i$, where $\alpha_i \in \mathbb{R}_{\geq 0}$. Then we have the estimate $n \leq 95 + 2 \sqrt{2056 + 63 \sum_{i\geq 6} \Delta_i \alpha_i}$. In particular, if $\alpha_i=0$ for all $i \geq 6$ then we get back the result from Theorem \ref{no simp theo}, part b). 
\end{remark}

\end{subsection}

\begin{subsection}{A lower bound for $t^\mathcal{A}_2$ and the Dirac Motzkin Conjecture}
In this subsection we study the so called \textit{Dirac Motzkin Conjecture}. In its classical form, it asserts that for a nontrivial arrangement $\mathcal{A}$ of $n$ \textit{straight} lines in $\mathbb{P}^2(\mathbb{R})$ one always has $$t^\mathcal{A}_2 \geq \Bigl \lfloor \frac{n}{2} \Bigr \rfloor.$$ This has been a famous open problem for a long time until in the paper \cite{greentao} said conjecture has been shown to be a theorem at least for \textit{sufficiently large} arrangements. However, the lower bounds given in the paper concerning the ``sufficiently large'' part are of double exponential order.  

We study this conjecture in the context of real \textit{pseudoline} arrangements whose characteristic polynomials have only real roots. 
This is motivated by the fact that there is an infinite family of such arrangements (denoted $\mathcal{R}(1)$ in \cite{Gruenbaum}) with $t^\mathcal{A}_2= \frac{|\mathcal{A}|}{2}$ for every $\mathcal{A}$ in the family. We remark that these are all linear simplicial arrangements. In the paper \cite{greentao}, the dual point configurations corresponding to arrangements in the family $\mathcal{R}(1)$ are the so called ``B\"or\"ocky examples'', denoted by $X_{2m}$ for $m \in \mathbb{N}_{\geq 3}$. Moreover, all arrangements from the family $\mathcal{R}(1)$ are \textit{(inductively) free}.

 By Lemma \ref{near pencil lemma}, every chamber of a nontrivial simplicial arrangement contains at most one double point. Therefore, the simpliciality of the line arrangements corresponding to the B\"or\"ocky examples may not seem surprising: simplicial arrangements could in general be expected to yield ``corner cases'' for the Dirac Motzkin Conjecture. 
 
 Besides the facts mentioned above, this point of view is supported by the observation that, apparently, the only known examples with $t^\mathcal{A}_2 < \frac{n}{2}$ are the simplicial arrangements $A(7,1), A(13,4)$. We observe that the characteristic polynomial of the first arrangement splits over $\mathbb{R}$ while for the second arrangement this is not the case. Moreover, the first arrangement is known as the ``Kelly-Moser configuration'' while the second one is known as the ``Crowe-McKee configuration'' (see \cite{kelly_moser}, \cite{crowe_mckee}).  
 
\begin{remark} \label{inf series bemerkung}
We note that (up to combinatorial isomorphism) the arrangements from the infinite family $\mathcal{R}(1)$ are completely characterized by their $t$-vectors. Indeed, the smallest arrangement in the family $\mathcal{R}(1)$ is the arrangement $\mathcal{A^\prime}:= A(6,1)$ which is characterized by the vector $t^\mathcal{A^\prime}=(3,4)$. Similarly, if $|\mathcal{A}| \geq 8$ then $\mathcal{A}$ belongs to $\mathcal{R}(1)$ if and only if there exists some $m \in \mathbb{N}_{\geq 4}$ such that $|\mathcal{A}|=2m$ and $t^\mathcal{A}_2=m, t^\mathcal{A}_3=\frac{m^2-m}{2}, t^\mathcal{A}_m=1$ while $t^\mathcal{A}_i=0$ for every $i \notin \lbrace 2,3,m \rbrace$. In the papers \cite{Cuntz}, \cite{Gruenbaum} the corresponding isomorphism class is denoted by $A(2m,1)$. The arrangements from the family $\mathcal{R}(1)$ also appear in the paper \cite{cuntz_muecksch}, which gives a classification of \textit{supersolvable} simplicial hyperplane arrangements (in arbitrary rank). Note that for linear line arrangements, supersolvability of $\mathcal{A}$ amounts to the fact that there is a single vertex of $\mathcal{A}$ which is connected (via lines of $\mathcal{A}$) to every other vertex.

Note also that for $n \geq 2$, one may add a suitable line to the arrangement $A(4n,1)$ to obtain a new simplicial arrangement denoted by $A(4n+1,1)$. The arrangements of type $A(4n+1,1)$ with $n \geq 2$ constitute another infinite family, which is denoted by $\mathcal{R}(2)$. Again, we refer to \cite{Gruenbaum} for more details. 
\end{remark} 
 
 Motivated by the above observations, we now state and prove the main result of this subsection, which provides a lower bound for $t^\mathcal{A}_2$. In the following, this will allow us to resolve the Dirac Motzkin Conjecture completely in the setup described above.
 
\begin{theorem} \label{t_2 quadratisch fuer faktorisierende arrangements}
If $n \geq 4$ then $t^\mathcal{A}_2 \geq 3 + \frac{(n-5)^2-4}{4m(\mathcal{A})-8}$.  
\end{theorem}

\begin{proof}
By Lemma \ref{basic comb t vector}, equation (1) one has $2 t^\mathcal{A}_3= \frac{n^2-n}{3} - \frac{2 t^\mathcal{A}_2}{3} - \sum_{i \geq 4} \frac{i^2-i}{3} t^\mathcal{A}_i$. Combining this with equation (2) and inequality (5) from said lemma and remembering Lemma \ref{poly lemma}, we obtain \begin{align*}
\frac{(n+1)^2}{4} \geq f^\mathcal{A}_2&=\frac{n^2-n+3}{3} + \frac{t^\mathcal{A}_2}{3} - \sum_{i \geq 4} \frac{i^2-4i+3}{3} t^\mathcal{A}_i \\
&\geq \frac{n^2-n+3}{3} + 1 + \sum_{i \geq 4} \frac{i-3}{3}t^\mathcal{A}_i - \sum_{i \geq 4} \frac{i^2-4i+3}{3} t^\mathcal{A}_i \\
&=\frac{n^2-n+6}{3} - \frac{1}{3} \sum_{i \geq 4} (i-2)(i-3) t^\mathcal{A}_i. 
\end{align*}

We conclude that $\sum_{i \geq 4} (i-2)(i-3) t^\mathcal{A}_i \geq \frac{n^2-10n+21}{4}=\frac{(n-5)^2}{4} - 1$. By definition, we have $i-2 \leq m(\mathcal{A}) - 2$ for every $i$ such that $t^\mathcal{A}_i > 0$. Using this, we obtain the following chain of inequalities: $$(m(\mathcal{A})-2) (t^\mathcal{A}_2 - 3)  \geq \sum_{i \geq 4} (i-2)(i-3) t^\mathcal{A}_i \geq \frac{(n-5)^2}{4} - 1.$$ Observe that $m(\mathcal{A} \geq 3$ because $n \geq 4$: indeed, if $m(\mathcal{A})=2$ then $f^\mathcal{A}_2=1+t^\mathcal{A}_2=1 + \binom{n}{2}>\frac{(n+1)^2}{4}$ for every $n \geq 4$. We conclude $t^\mathcal{A}_2 \geq 3 + \frac{(n-5)^2-4}{4m(\mathcal{A})-8}$, finishing the proof.  
\end{proof} 

\begin{remark} \label{lower bound t3 remark}
If $\mathcal{A}$ is simplicial, then one also has the following inequalities: \\
$t^\mathcal{A}_3 + \frac{2 t^\mathcal{A}_4+t^\mathcal{A}_5}{m(\mathcal{A})} \geq 4 + \frac{(n-5)^2-4}{4m(\mathcal{A})} + \frac{\sum_{i \geq 7} (i-6)t^\mathcal{A}_i}{m(\mathcal{A})}\geq 4 + \frac{(n-5)^2-4}{4m(\mathcal{A})}.$ This will turn out to be useful in the following subsection.
\end{remark}

The last theorem is enough to resolve the Dirac Motzkin Conjecture for arrangements whose characteristic polynomials split over $\mathbb{R}$:

\begin{theorem} \label{dirac theorem}
The following statements hold:\\
a) We have $t^\mathcal{A}_2 \geq \lfloor \frac{n}{2} \rfloor$. \\
b) If $t^\mathcal{A}_2=\lfloor \frac{n}{2} \rfloor$, then $\mathcal{A}$ is simplicial. More precisely: if $n$ is even, then $\mathcal{A}$ belongs to the infinite family $\mathcal{R}(1)$. If $n$ is odd, then $\mathcal{A}$ is the Kelly-Moser example.
\end{theorem}

\begin{proof}
a) We consider three cases, corresponding to size and multiplicity of the given arrangement: \\
Case i): Assume that $3 \leq n \leq 7$.

 Then by relation (5) from Lemma \ref{basic comb t vector} we have $t^\mathcal{A}_2 \geq 3 \geq \lfloor \frac{n}{2} \rfloor$. \\
Case ii): Assume that $n \geq 8$ and $m(\mathcal{A}) \geq \frac{n}{2}$. 

Then there exists $j \geq \frac{n}{2} \geq 4$ such that $t^\mathcal{A}_j>0$ and relation (5) from Lemma \ref{basic comb t vector} yields $t^\mathcal{A}_2 \geq 3 + \sum_{i \geq 4} (i-3) t^\mathcal{A}_i \geq 3 + (j - 3) t^\mathcal{A}_j  \geq 3 + \frac{n}{2} - 3 = \frac{n}{2} \geq  \lfloor \frac{n}{2} \rfloor.$ \\
Case iii): Assume that $n \geq 8$ and $m(\mathcal{A}) < \frac{n}{2}$. 

By Theorem \ref{t_2 quadratisch fuer faktorisierende arrangements}, we have $$t^\mathcal{A}_2 \geq 3 + \frac{(n-5)^2-4}{4m(\mathcal{A})-8} \geq 3 + \frac{(n-5)^2-4}{4\frac{n-1}{2}-8}=\frac{1}{2} \left( n + 1 - \frac{4}{n-5} \right).$$ Clearly, for $n \geq 9$ one has $t^\mathcal{A}_2 \geq \frac{1}{2} (n + 1 - \frac{4}{n-5}) \geq \frac{n}{2} \geq  \lfloor \frac{n}{2} \rfloor$. For $n=8$ we obtain $\frac{23}{6} \leq t^\mathcal{A}_2 \in \mathbb{N} $, which implies $t^\mathcal{A}_2 \geq 4=\frac{n}{2} \geq \lfloor \frac{n}{2} \rfloor $. \\
b) Let $\mathcal{A}$ be an arrangement such that $t^\mathcal{A}_2 = \lfloor \frac{n}{2} \rfloor$. We consider two cases, corresponding to the parity of $n$:\\
Case i): Assume that $n$ is odd.

Then we have $t^\mathcal{A}_2=\lfloor \frac{n}{2} \rfloor=\frac{n-1}{2} < \frac{n}{2}$. The proof of part a) shows that we necessarily have $3 \leq n \leq 7$. Indeed, for $n \geq 8$ one always has $t^\mathcal{A}_2 \geq \frac{n}{2}$. Moreover, for $3 \leq n \leq 6$ we also have $t^\mathcal{A}_2 \geq 3 \geq \frac{n}{2}$. We conclude that $n=7$ and $t^\mathcal{A}_2=3$. In particular, $\mathcal{A}$ is simplicial as we have equality in relation (5) from Lemma \ref{basic comb t vector}. Using \cite{Cuntz}, we conclude that $\mathcal{A}$ is the Kelly-Moser example. 
Case ii): Assume that $n$ is even.

Then we have $t^\mathcal{A}_2=\frac{n}{2}$. First assume that $n \geq 8$.
We show that the multiplicity of $\mathcal{A}$ is precisely $\frac{n}{2}$. If not, then $m(\mathcal{A}) > \frac{n}{2}$ or $m(\mathcal{A})< \frac{n}{2}$. 

So assume that $m(\mathcal{A})> \frac{n}{2}$. Then $t^\mathcal{A}_2 \geq 3 +  \sum_{i \geq 4} (i-3)t^\mathcal{A}_i >3+(\frac{n}{2}-3)=\frac{n}{2}$, a contradiction. 

Now assume that $m(\mathcal{A}) < \frac{n}{2}$. Then $i-2 \leq m(\mathcal{A})-2 \leq \frac{n-2}{2} - 2=\frac{n-6}{2}$ for every $i$ such that $t^\mathcal{A}_i>0$. Using that $\frac{n-6}{2}=t^\mathcal{A}_2 - 3$ and remembering the proof of Theorem \ref{t_2 quadratisch fuer faktorisierende arrangements}, we conclude that $$\frac{(n-6)^2}{4}=\frac{n-6}{2}(t^\mathcal{A}_2 - 3)  \geq \sum_{i \geq 4} (i-2)(i-3) t^\mathcal{A}_i \geq \frac{(n-5)^2}{4} - 1.$$
Clearly, this is impossible for $n \geq 8$. We obtain $m(\mathcal{A})=\frac{n}{2}$. Using this, relation (5) from Lemma \ref{basic comb t vector} yields $t^\mathcal{A}_{\frac{n}{2}}=1$ while $t^\mathcal{A}_i =0$ for $i \notin \lbrace 2,3, \frac{n}{2} \rbrace$. Using equation (1) from Lemma \ref{basic comb t vector}, it follows that $t^\mathcal{A}_3=\frac{n^2-2n}{8}$. By Remark \ref{inf series bemerkung}, we may conclude that $\mathcal{A}$ belongs to the infinite family $\mathcal{R}(1)$. 

It remains to consider the cases $n=4$ and $n=6$. For $n=4$, we obtain $2=\frac{n}{2}=	t^\mathcal{A}_2 \geq 3$, which is impossible. If $n=6$, then $t^\mathcal{A}_2=\frac{n}{2}=3$. In particular, we have equality in relation (5) from Lemma \ref{basic comb t vector}. Using the enumeration in \cite{Cuntz}, it follows that $\mathcal{A}$ is the arrangement $A(6,1)$, which is the smallest arrangement from the family $\mathcal{R}(1)$. This completes the proof. 
\end{proof}

\begin{remark}
We remark that there are non-simplicial arrangements for which we have equality in the Dirac Motzkin Conjecture: if one removes a suitable line from the (simplicial) arrangement $A(13,4)$, then one obtains a non-simplicial arrangement $\mathcal{A}$ with $t^\mathcal{A}_2 = \frac{n}{2}$. Note also that the corresponding characteristic polynomial $\chi(\mathcal{A},t)$ has a non-real root, in accordance with Theorem \ref{dirac theorem}. However, as pointed out already, the only two known arrangements with $t^\mathcal{A}_2 < \frac{n}{2}$ are both simplicial.
\end{remark}

Theorem \ref{dirac theorem} and the above remark lead us to the following conjecture, which closes this subsection:

\begin{theoremb}
Let $\mathcal{A}$ be an arbitrary arrangement consisting of $n$ pseudolines (where $\chi(\mathcal{A},t)$ may or may not split over $\mathbb{R}$). If $t^\mathcal{A}_2 < \frac{n}{2}$, then $\mathcal{A}$ is necessarily simplicial.  
\end{theoremb}

\end{subsection}

\begin{subsection}{Simplicial arrangements with multiplicity at most six}
In this subsection, $\mathcal{A}$ will always denote a nontrivial simplicial arrangement with $m(\mathcal{A}) \leq 6$ (and of course such that $\chi(\mathcal{A},t)$ splits over $\mathbb{R}$).

We prove that $m(\mathcal{A}) \leq 5$ implies $n \leq 40$. Moreover, if $m(\mathcal{A}) \leq 4$, then $n \leq 16$ which gives a classification result (using the enumeration presented in \cite{Cuntz}). In both cases, there are only finitely many isomorphism classes of arrangements of the respective type.
If $m(\mathcal{A}) \leq 6$ and if $t^\mathcal{A}_2$ is not too large compared to $t^\mathcal{A}_3$, then we can also prove that there are only finitely many possibilities for the isomorphism class of $\mathcal{A}$.  
Finally, we show that the validity of an old conjecture stated in \cite{erd} is related to the following theorem, which gives asymptotically optimal estimates for $t^\mathcal{A}_2, t^\mathcal{A}_3, t^\mathcal{A}_6$, and which is considered the main result of this subsection.

\begin{theorem} \label{sechser struc}
We have the following estimates: \begin{align}
\frac{(n-5)^2 +44}{16} &\leq t^\mathcal{A}_2 \leq \frac{(n+1)^2}{16}, \\
\frac{n^2-22n+185}{24} &\leq t^\mathcal{A}_3 \leq \frac{n^2+116 n -597}{24}, \\
 t^\mathcal{A}_4 &+ t^\mathcal{A}_5 \leq \frac{3}{2} n - \frac{17}{2}, \\
\frac{n^2-46n+225}{48}  &\leq t^\mathcal{A}_6 \leq \frac{n^2+2n-47}{48}.
\end{align} 
\end{theorem}

\begin{proof}
By Remark \ref{near pencil remark}, Lemma \ref{poly lemma} and Lemma \ref{near pencil lemma}, part b), we have $t^\mathcal{A}_2=3 + t^\mathcal{A}_4 + 2 t^\mathcal{A}_5 + 3 t^\mathcal{A}_6 \leq \frac{(n+1)^2}{16}$. We conclude $t^\mathcal{A}_6 \leq \frac{n^2}{48} + \frac{n}{24} - \frac{t^\mathcal{A}_4}{3} - \frac{2 t^\mathcal{A}_5}{3} - \frac{47}{48}$. Using Theorem \ref{t_2 quadratisch fuer faktorisierende arrangements}, we obtain $t^\mathcal{A}_2 \geq  3+\frac{(n-5)^2-4}{16}$, proving inequality (13).

As  $3 t^\mathcal{A}_3= \binom{n}{2} -t^\mathcal{A}_2 -\sum_{i \geq 4} \binom{i}{2} t^\mathcal{A}_i$ and $\frac{(n+1)^2}{4} \geq f^\mathcal{A}_2=2 \left( f^\mathcal{A}_0 - 1 \right)$, we may deduce that 
$  \frac{n^2}{48} - \frac{5 n}{24} + \frac{7}{16} - \frac{t^\mathcal{A}_5}{2} - \frac{t^\mathcal{A}_4}{6} \leq t^\mathcal{A}_6$.

Now we combine the upper and lower bounds for $t^\mathcal{A}_6$ and obtain the following estimate: \begin{align}
\frac{n^2}{48} - \frac{5 n}{24} + \frac{7}{16} - \frac{t^\mathcal{A}_5}{2} - \frac{t^\mathcal{A}_4}{6} \leq t^\mathcal{A}_6 \leq \frac{n^2}{48} + \frac{n}{24} - \frac{t^\mathcal{A}_4}{3} - \frac{2 t^\mathcal{A}_5}{3} - \frac{47}{48}. 
\end{align}
This implies inequality (15). Inequality (16) now follows from (17) using (15):
we have $\frac{n^2}{48} - \frac{5 n}{24} + \frac{7}{16} - \frac{t^\mathcal{A}_4+t^\mathcal{A}_5}{2} \leq \frac{n^2}{48} - \frac{5 n}{24} + \frac{7}{16} - \frac{t^\mathcal{A}_5}{2} - \frac{t^\mathcal{A}_4}{6}$ and as $t^\mathcal{A}_4 + t^\mathcal{A}_5 \leq \frac{3 n}{2} - \frac{17}{2} $ we conclude $\frac{n^2-46n+225}{48} =\frac{n^2}{48} - \frac{5 n}{24} + \frac{7}{16} - \frac{3 n}{4} + \frac{17}{4} \leq t^\mathcal{A}_6$. Moreover, $t^\mathcal{A}_6 \leq \frac{n^2}{48} + \frac{n}{24} - \frac{t^\mathcal{A}_4}{3} - \frac{2 t^\mathcal{A}_5}{3} - \frac{47}{48} \leq \frac{n^2}{48} + \frac{n}{24} - \frac{47}{48}$ because $t^\mathcal{A}_i \geq 0$ for all $i \geq 2$.

 Finally, the lower bound in (14) follows from Remark \ref{lower bound t3 remark}. The upper bound follows from (13),(15) and (16) using equation (1) from Lemma \ref{basic comb t vector}. This completes the proof.
\end{proof}

We now draw some conclusions from Theorem \ref{sechser struc}.

\begin{corollary} \label{hoechstens 4er}
a) If $m(\mathcal{A}) \leq 4$ then $n \leq 16$. In particular, we have a complete list of such arrangements.  \\
b) If $\mathcal{A}$ is stretchable and $m(\mathcal{A}) \leq 4$, then $\mathcal{A}$ admits a crystallographic rootset. Moreover, the arrangement $\mathcal{A}$ may be obtained as a subarrangement of the arrangement $A(13,2)$ depicted in Figure \ref{a_13_2 bild}.
\end{corollary}

\begin{proof}
a) We have $t^\mathcal{A}_i=0$ for $i>4$. Therefore, by equation (5) from Lemma \ref{basic comb t vector} it follows $t^\mathcal{A}_2=3 + t^\mathcal{A}_4$. Consequently, by Theorem \ref{sechser struc} we obtain the upper bound $t^\mathcal{A}_2 \leq 3 + \frac{3}{2} n - \frac{17}{2}$. Using Theorem \ref{t_2 quadratisch fuer faktorisierende arrangements}, we obtain $t^\mathcal{A}_2 \geq 3 + \frac{(n-5)^2-4}{8}$. It follows $\frac{(n-10)n+45}{8} \leq t^\mathcal{A}_2 \leq \frac{3}{2} n - \frac{11}{2}$ which implies $1 \leq n \leq 16$. This proves the first claim and implies the second (using the results in \cite{Cuntz}). \\
b) This follows from part a) by inspecting the catalogue provided in \cite{Cuntz}.
\end{proof}

\begin{remark}
a) As every crystallographic arrangement is inductively free (see \cite{Cuntz2}), the last corollary shows that for a linear simplicial arrangement $\mathcal{A}$ with  $m(\mathcal{A}) \leq 4$, the notions of being free and being inductively free coincide. \\
b) Observe that the arrangement $A(13,2)$ is obtained as a restriction of the reflection arrangement of type $F_4$. 
\end{remark}

We continue with the situation where $m(\mathcal{A}) \leq 5$. We also obtain a classification result in this case if the number of double points is not too large compared with the number of triple points.

\begin{corollary} \label{hoechstens fuenfer}
Assume that $m(\mathcal{A}) \leq 5$. Then the following statements hold: \\
a) We have $n \leq 40$. \\
b) If $t^\mathcal{A}_2 \leq \frac{13}{16} t^\mathcal{A}_3$, then $n \leq 27$. In particular, we have a complete list of such arrangements.  
\end{corollary}

\begin{proof}
a) By assumption and Theorem \ref{sechser struc} we have $\frac{n^2-46n+225}{48}  \leq t^\mathcal{A}_6=0$. This implies $n \leq 40$. \\
b) First, we observe that $t^\mathcal{A}_2 \geq  \frac{n^2-10n+57}{12}$, by Theorem \ref{t_2 quadratisch fuer faktorisierende arrangements}. As $\frac{13}{16} t^\mathcal{A}_2 \leq t^\mathcal{A}_3$ we may invoke Lemma \ref{t2+t3} to arrive at the inequality $\frac{n^2-10n+57}{12} \leq t^\mathcal{A}_2 \leq \frac{8(n^2-n+30)}{135}$. It follows that we necessarily have $n \leq 27$. Using the results obtained in \cite{Cuntz}, this completes the proof.
\end{proof}

For arrangements having multiplicity six, we can only prove finiteness results if the number of double points is not too large. More precisely, we will show that for $\epsilon>0$ there are only finitely many isomorphism classes of simplicial arrangements $\mathcal{A}$ with $ t^\mathcal{A}_2 \leq \frac{24}{16+\epsilon} t^\mathcal{A}_3$.

\begin{corollary} \label{hoechstens 6er}
Assume that $ t^\mathcal{A}_2 \leq \frac{24}{16+\epsilon} t^\mathcal{A}_3$ for some $\epsilon>0$. Then we have $n \leq \frac{2 \sqrt{254016-11\epsilon^2- 144 \epsilon} +5 \epsilon + 1008}{\epsilon}$ and $\epsilon \leq \frac{72}{11} (6 \sqrt{15} - 1)$. 
\end{corollary}

\begin{proof}
By Theorem \ref{sechser struc} we have $\frac{(n-5)^2 +44}{16} \leq t^\mathcal{A}_2$ and $t^\mathcal{A}_3 \leq \frac{n^2+116 n -597}{24}$. By assumption we obtain $\frac{n^2+116n-597}{16+\epsilon} \geq \frac{(n-5)^2 +44}{16}$. This gives the result.
\end{proof}

Unfortunately, we are not able to prove an absolute bound for $n$ if $t^\mathcal{A}_2 > t^\mathcal{A}_3$. However, let $M$ denote the set of isomorphism classes of linear simplicial arrangements whose multiplicities are at most six and whose characteristic polynomials split over $\mathbb{R}$. We can then relate the cardinality of $M$ to the following conjecture stated in \cite{erd}: \begin{theoremb}
Let $5 \leq k \in \mathbb{N}$ be a natural number and let $\mathfrak{A}_k$ denote the set of all isomorphism classes of line arrangements in $\mathbb{P}^2(\mathbb{R})$ having multiplicity at most $k$. Then for any sequence of arrangements $\left( \mathcal{A}_\nu \right)_{\nu \in \mathbb{N}}$ such that $\mathcal{A}_\nu \in \mathfrak{A}_k$ and  $\lim_{\nu \to \infty} |\mathcal{A}_\nu| = \infty$  we have \begin{align*}
\lim_{\nu \to \infty} \frac{t^\mathcal{A_\nu}_k}{|\mathcal{A}_\nu|^2}=0.
\end{align*}  
\end{theoremb}

Then Theorem \ref{sechser struc} yields the following corollary, which closes this paper.

\begin{corollary}
If $|M|=\infty$ then the above conjecture is false for $k=6$.
\end{corollary}

\begin{proof}
By Theorem \ref{sechser struc} we have $\frac{|\mathcal{A}|^2-46|\mathcal{A}|+225}{48}  \leq t^\mathcal{A}_6 \leq \frac{|\mathcal{A}|^2+2|\mathcal{A}|-47}{48}$, if $\mathcal{A} \in M$. So if $|M|=\infty$ we find a sequence $\left( \mathcal{A}_\nu \right)_{\nu \in \mathbb{N}}$ such that $\lim_{\nu \to \infty} |\mathcal{A}_\nu| = \infty$ and $\mathcal{A}_\nu \in M$ for every $\nu \in \mathbb{N}$. But then $\lim_{\nu \to \infty} \frac{t^\mathcal{A_\nu}_k}{|\mathcal{A}_\nu|^2}= \frac{1}{48}>0$.  
\end{proof}

\end{subsection}

\end{section}
\def\cprime{$'$}
\providecommand{\bysame}{\leavevmode\hbox to3em{\hrulefill}\thinspace}
\providecommand{\MR}{\relax\ifhmode\unskip\space\fi MR }
\providecommand{\MRhref}[2]{%
  \href{http://www.ams.org/mathscinet-getitem?mr=#1}{#2}
}
\providecommand{\href}[2]{#2}

\end{document}